\documentclass[11pt,a4paper]{article}
\usepackage{amsmath}
\usepackage{amssymb}
\usepackage{amsfonts}
\usepackage{amsthm}
\usepackage{relsize}
\usepackage{setspace}
\usepackage{geometry}
\usepackage{enumerate}
\usepackage{url}
\usepackage{xspace}

\newtheorem{thm}{Theorem}
\newtheorem{lem}[thm]{Lemma}

\newtheorem{cor}[thm]{Corollary}

\theoremstyle{definition}
\newtheorem{defn}[thm]{Definition}

\newcommand{\defbold}{\textbf}

\newcommand{\inv}{^{-1}}

\newcommand{\tdlc}{t.d.l.c.\@\xspace}

\newcommand{\Cau}{\mathrm{CF}}

\newcommand{\id}{\mathrm{id}}

\newcommand{\Comm}{\mathrm{Comm}}

\newcommand{\bQ}{\mathbb{Q}}

\newcommand{\bZ}{\mathbb{Z}}

\newcommand{\mcC}{\mathcal{C}}

\newcommand{\mcF}{\mathcal{F}}

\begin{document}

\title{The profinite completion of a group localised at a subgroup}

\author{Colin D. Reid}

\maketitle

\begin{abstract}Let $G$ be a group and let $K$ be a commensurated subgroup of $G$.  Then there is a totally disconnected, locally compact (\tdlc) group $\hat{G}_K$ that contains the profinite completion of $K$ as an open compact subgroup and also contains $G$ (modulo the finite residual of $K$) as a dense subgroup.  Moreover, given an arbitrary group $G$, then every \tdlc group containing an image of $G$ as a dense subgroup can be realised as a quotient of $\hat{G}_K$ for some commensurated subgroup $K$.\end{abstract}

The profinite completion is a natural way of embedding a residually finite group in a profinite group.  Some properties of the original group can usefully be analysed through its profinite completion, since profinite groups are well-behaved with respect to finite images, and as compact groups, they are amenable to analytic methods.  (See \cite{Seg} for some examples of profinite methods in group theory.)  This approach can be generalised to other classes of group that may not themselves have interesting finite images, but nevertheless have a residually finite subgroup that is commensurated under conjugation, for instance $\mathrm{SL}_n(\bZ)$ as a subgroup of $\mathrm{SL}_n(\bQ)$.  In this setting, we can obtain a relative or localised profinite completion that is a totally disconnected, locally compact (\tdlc) group $G$, such that the image of the commensurated subgroup is dense in an open compact subgroup of $G$.  This approach has been taken, for instance, in \cite{Sch} and \cite{SW}.

The purpose of this note is to call attention to what might be considered the universal construction of this kind, in a sense that we will make precise.  There are no surprises in either  the statement or the proof of Theorem \ref{mainthm}, but it does not appear to be widely known in this form.  This is in sharp contrast to the special case of the profinite completion, which is by now very much a textbook construction: for instance it appears in the first chapter of \cite{Wil}, and pro-$\mcC$ completions for a class of finite groups $\mcC$ feature prominently in \cite{Rib}.  (After writing an earlier version of this note, I was made aware of a paper by V. Belyaev that gives an equivalent construction - see \cite{Bel} Theorem 7.1.  My thanks to Aleksander Iwanow for pointing this out.)

Throughout, the notation $\overline{H}$ denotes the topological closure of $H$ (the ambient topological space will be clear from the context).

\begin{thm}\label{mainthm}Let $G$ be a topological group and let $K$ be a closed subgroup of $G$ such that $K \cap gKg\inv$ has finite index in $K$ for all $g \in G$.  Then there is an essentially unique group homomorphism $\theta: G \rightarrow \hat{G}_K$, the \defbold{localised profinite completion of $G$ at $K$}, where $\hat{G}_K$ is a topological group, such that the following properties hold:
\begin{enumerate}[(i)]
\item The image of $\theta$ is dense, the restriction $\theta|_K$ of $\theta$ to $K$ is continuous, and $\overline{\theta(K)}$ is an open profinite subgroup of $\hat{G}_K$.
\item Let $R$ be a topological group and let $\phi: G \rightarrow R$ be a group homomorphism.  Suppose that $\phi|_K$ is continuous and $\overline{\phi(K)}$ is profinite.  Then there is a unique continuous homomorphism $\psi: \hat{G}_K \rightarrow R$ such that $\phi = \psi\theta$.\end{enumerate}\end{thm}

\begin{defn}Let $G$ be a topological group.  The (topological) \defbold{profinite completion} ${G \rightarrow \hat{G}}$ of $G$ is the limit of the inverse system formed by the quotients $G/N$ as $N$ ranges over the closed normal subgroups of $G$ of finite index.  With this definition, the profinite completion is stable in the sense that 
\[\hat{\hat{G}} \cong \hat{G}\]
for any topological group $G$.  The image of the profinite completion is dense and the kernel is the intersection of all closed subgroups of $G$ of finite index.

Two subgroups $H$ and $K$ of $G$ are \defbold{commensurate} if $H \cap K$ has finite index in both $H$ and $K$.  The \defbold{commensurator} $\Comm_G(K)$ of $K$ in $G$ is the set of elements $g \in G$ such that $gKg\inv$ is commensurate with $K$.  Note that $\Comm_G(K)$ is a subgroup of $G$ (although in general it is not closed even if $K$ is closed: see for instance \cite{Rei}).  Also, notice that if $H$ is a subgroup of $G$ and $g \in G$ is such that $g$ commensurates both $H$ and $K$, then $g$ commensurates $H \cap K$ (since the commensurability class of $H \cap K$ is determined by the commensurability classes of $H$ and $K$).\end{defn}

In the more general case of a topological group $G$ and a closed subgroup $K$ of $G$ that is not necessarily commensurated by $G$, we define the localised profinite completion of $G$ at $K$ to be the localised profinite completion of $H$ at $K$, where $H$ is the commensurator of $K$ in $G$ equipped with the subspace topology, and write $\hat{G}_K := \hat{H}_K$.

Here are some observations on the localised profinite completion, all of which will follow easily from the proof of Theorem \ref{mainthm}.
\begin{cor}Let $G$ be a topological group, let $K$ be a closed subgroup of $G$ and let $\theta$ be the localised profinite completion of $G$ at $K$.  Given a subgroup $H$ of $\Comm_G(K)$, write $\tilde{H} = \overline{\theta(H)}$.
\begin{enumerate}[(i)]
\item Suppose $G$ is discrete.  Then taken collectively, the localised profinite completions of $G$ at its commensurated subgroups account for every homomorphism from $G$ to a \tdlc group.

Specifically, let $\phi: G \rightarrow R$ be such a homomorphism, let $U$ be an open compact subgroup of $R$, and set $K = \phi^{-1}(U)$.    Then $G$ commensurates $K$, since $\phi(G)$ commensurates $U$; moreover, there is a unique continuous homomorphism $\psi: \hat{G}_K \rightarrow R$ such that $\phi = \psi\theta$.
\item Let $H$ be a subgroup of $\Comm_G(K)$ that contains $K$.  Then the map from $H$ to $\tilde{H}$ induced by $\theta$ is the localised profinite completion of $H$ at $K$.
\item The induced map from $K$ to $\tilde{K}$ is the profinite completion of $K$.  In particular $\hat{K}_K = \hat{K}$.
\item The kernel of $\theta$ is the intersection of all closed subgroups of $K$ of finite index.  Indeed $\theta^{-1}(\tilde{N})=N$ for every closed subgroup $N$ of $K$ of finite index.
\item Let $K$ and $L$ be closed subgroups of $G$ such that $K$ is commensurate to $L$.  Then $\hat{G}_K = \hat{G}_L$; indeed the localised profinite completion maps of $G$ at $K$ and at $L$ are identical.
\item The homomorphism $\theta$ is continuous if and only if $K$ is open in $G$.
\item Suppose that $G$ is discrete and $K$ is a commensurated subgroup of $G$; let $R$ be the core of $\tilde{K}$ in $\hat{G}_K$.  Let $G /\!/ K$ and $\tau_{G,K}: G \rightarrow G/\!/K$ be as defined in \cite{SW}.  Then there is a topological quotient homomorphism $\eta: \hat{G}_K \rightarrow G /\!/ K$ with kernel $R$ such that $\tau_{G,K} = \eta\theta$.
\item Suppose that $G$ is a \tdlc group and $K$ is a profinite subgroup of $G$.  Then $\theta(K)$ is closed in $\hat{G}_K$ and $\theta$ is an isomorphism of abstract groups such that $\theta\inv$ is continuous.  In this case we have $\hat{G}_K \cong G_{(K)}$, where $G_{(K)}$ is the localisation of $G$ at $K$ defined in \cite{Rei}.  In particular, if $K$ is open then the localised profinite completion of $G$ at $K$ is an isomorphism of topological groups.\end{enumerate}\end{cor}

\begin{defn}Let $G$ be a topological group and let $K$ be a subgroup of $G$.  Let $\mcF(K)$ be the set of closed subgroups of $K$ of finite index.  A \defbold{filter} of $G$ is a set of subsets of $G$ such that 
\[F_1 \in f, F_1 \subseteq F_2 \subseteq H \Rightarrow F_2 \in f\]
 and
\[F_1,F_2 \in f \Rightarrow F_1 \cap F_2 \in f.\]
Note that every set $\Sigma$ of subsets of $G$ is contained in a unique smallest filter, which we call the filter generated by $\Sigma$.  Let $f$ be a filter of $G$.  Say $f$ is \defbold{proper} if $\emptyset \not\in f$.  Say $f$ \defbold{left Cauchy relative to $K$} if $f$ is proper and for every $N \in \mcF(K)$, $f$ contains some (and hence exactly one) left coset of $N$, and $f$ is generated as a filter by left cosets of elements of $\mcF(K)$.  Two proper filters $f$ and $f'$ are \defbold{left equivalent relative to $K$} if their intersection is left Cauchy relative to $K$.  The definitions of `right Cauchy relative to $K$' and `right equivalent relative to $K$' are analogous, with right cosets in place of left cosets.  We say $f$ is \defbold{Cauchy relative to $K$} if it is both left and right Cauchy relative to $K$, and say $f_1$ and $f_2$ are \defbold{equivalent relative to $K$} if they are both left and right equivalent.  In particular, any principal filter of $G$, that is, the filter $f_g$ consisting of all subsets of $G$ containing $g$ for some fixed $g \in G$, is Cauchy.

Write $\Cau(G,K)$ for the set of Cauchy filters of $G$ relative to $K$, modulo equivalence relative to $K$.  In essence, an equivalence class of Cauchy filters relative to $K$ is just a consistent choice of left and right cosets of elements of $\mcF(K)$.  To avoid over-complicating the notation, we will tacitly take representatives and simply regard elements of $\Cau(G,K)$ as being filters of $G$ when there is no ambiguity in doing so.  For instance, if $f$ is an element of $\Cau(G,K)$, the statement `$f$ contains $gK$' should be understood to mean that $gK$ is an element of some (equivalently, every) representative of $f$.\end{defn}

\begin{lem}\label{leftrightlem}Let $G$ be a group and let $K$ be a subgroup of $G$.  Suppose that $gKg\inv$ is commensurate with $K$ for all $g \in G$.  Let $N$ be a subgroup of $K$ of finite index and let $g \in G$.  Then there is a subgroup $M$ of $N$, such that $M$ is an intersection of finitely many conjugates of $N$, and such that for all $h \in G$, the set $gN \cap Nh$ is a union of left cosets of $M$.\end{lem}

\begin{proof}Let $R = N \cap gNg\inv$.  Then for all $h \in G$ we have $R(gN \cap Nh) = gN \cap Nh$, so $gN \cap Nh$ is a union of right cosets of $R$ in $G$.  Now the right cosets of $R$ in $G$ that are subsets of $gN$ are exactly those of the form $Rtg$ for $t \in gNg\inv$; since $R$ has finite index in $gNg\inv$, only finitely many such right cosets of $R$ exist.  Hence the set $\{gN \cap Nh \mid h \in G\}$ is finite: let $h_1,\dots,h_n \in G$ be such that
\[ \{gN \cap Nh \mid h \in G\} = \{gN \cap Nh_1, \dots, gN \cap Nh_n\}.\]
Now let $M = N \cap \bigcap^n_{i=1} h\inv_iNh_i$.  We see that $(gN \cap Nh)M = gN \cap Nh$ for all $h \in G$, so $gN \cap Nh$ is a union of left cosets of $M$.\end{proof}

\begin{lem}Let $G$ be a topological group and let $K$ be a subgroup of $G$.  Let $f$ be a filter of $G$.  Suppose that $gKg\inv$ is commensurate with $K$ for all $g \in G$.  Then $f$ is left Cauchy relative to $K$ if and only if $f$ is right Cauchy relative to $K$.  Two proper filters are left equivalent relative to $K$ if and only if they are right equivalent.\end{lem}

\begin{proof}For this proof, the Cauchy property and equivalence are understood to be relative to $K$.  Suppose $f$ is left Cauchy and let $N \in \mcF(K)$.  Then $f$ contains some left coset $gN$ of $N$.  Thus by Lemma \ref{leftrightlem} there is a subgroup $M$ of $N$, such that $M$ is an intersection of finitely many conjugates of $N$, so in particular $M$ is a closed subgroup of $K$ of finite index, and such that for all $h \in G$, the set $gN \cap Nh$ is a union of left cosets of $M$.  Thus the partition $\{hM \mid h \in gN\}$ of $gN$ is a subdivision of the partition $\{gN \cap Nh \mid h \in H\}$.  Since $f$ is left Cauchy, there is some left coset $kM$ of $M$ contained in $f$, and since $\emptyset \not\in f$ we have $kM \subseteq gN$.  Hence $kM \subseteq gN \cap Nh \subseteq Nh$ for some $h \in G$, so $Nh \in f$.  As $N$ was an arbitrary element of $\mcF(K)$, we conclude that $f$ is right Cauchy.  The converse holds by symmetry.

Let $f_1$ and $f_2$ be proper filters of $H$.  Then $f_1 \cap f_2$ is left Cauchy if and only if it is right Cauchy, and so $f_1$ and $f_2$ are left equivalent if and only if they are right equivalent.\end{proof}

\begin{proof}[Proof of Theorem \ref{mainthm}]

We construct $\hat{G}_K$ and $\theta$ as follows:

The underlying set of $\hat{G}_K$ is $\Cau(G,K)$.  The function $\theta: G \rightarrow \hat{G}_K$ is given by setting $\theta(g)$ to be the equivalence class of the principal filter generated by $g$.

Let $f_1, f_2 \in \hat{G}_K$, let $N \in \mcF(K)$.  Suppose $g_2N \in f_2$, and let $M = N \cap g_2Ng\inv_2$.  Then $f_1$ contains a left coset of $M$, say $g_1M$.  Now set $F_N = g_1Mg_2N = g_1g_2N$.  Define $f_1f_2$ to be the filter generated by $\{F_N \mid N \in \mcF(K)\}$.  We see that given $N_1, N_2 \in \mcF(K)$ such that $N_1 \le N_2$, we have $F_{N_1} \subseteq F_{N_2}$, so $f_1f_2$ is a proper filter; it is (left) Cauchy relative to $K$ by construction.  Thus $f_1f_2 \in \hat{G}_K$.

It is clear that the binary operation we have defined on $\hat{G}_K$ extends the operation on $\theta(G)$ induced by multiplication in $G$ and that $f\theta(1) = \theta(1)f = f$ for all $f \in \hat{G}_K$.  (From now on we will write $1$ in place of $\theta(1)$ if there is no danger of ambiguity.)  Given $f \in \hat{G}_K$, we obtain $f\inv$ as the equivalence class of the filter generated by $\{Nh\inv \mid N \in \mcF(K), h \in H, hN \in f\}$; it is easily seen that $f\inv$ is well-defined and that $ff\inv = f\inv f = 1$.  Finally, multiplication is associative: given $f_1,f_2,f_3 \in \hat{G}_K$, the products $(f_1f_2)f_3$ and $f_1(f_2f_3)$ are both characterised as being the element $f$ such that, for all $A_i \in f_i$, all $N \in \mcF(K)$ and all $g \in G$, if $A_1A_2A_3 \subseteq gN$ then $gN \in f$.  So $\hat{G}_K$ is a group and $\theta$ is a homomorphism.

Given $g \in G$ and $N \in \mcF(K)$, define $B(gN) := \{[f] \in \hat{G}_K \mid gN \in f\}$.  Define $B := \{ B(gN) \mid g \in G, N \in \mcF(K)\}$.  We claim that $B$ is a base for a group topology of $\hat{G}_K$.  It suffices to show that the preimage $P$ of $B(gN)$ under the map $(x,y) \mapsto xy\inv$ is a union of basic open sets in the product topology.  Indeed, we have the following:
\[ P = \bigcup_{h \in G}\{(f_1,f_2) \mid gNh \in f_1, Nh \in f_2\},\]
which is a union of basic open sets since $gNh$ and $Nh\inv$ can both be decomposed into left cosets of elements of $\mcF(K)$ using Lemma \ref{leftrightlem}.

Every basic open set contains $\theta(g)$ for some $g \in G$, so $\theta$ has dense image.  Basic open sets are in fact clopen, since $\{B(gN) \mid g \in G\}$ is a partition of $|\hat{G}_K$ for any $g \in G$ and $N \in \mcF(K)$.  Indeed, we see that $B(gN) = \overline{\theta(gN)}$ for any $g \in G$ and $N \in \mcF(K)$, and also that $\theta^{-1}(\overline{\theta(gN)}) = gN$.

Observe that there is a natural map $\alpha$ from $\overline{\theta(K)}$ to $\hat{K}$: given a closed normal subgroup $N$ of $K$ of finite index and $f \in \overline{\theta(K)}$, we set the $G/N$-entry of $\alpha(f)$ to be the unique element $gN$ of $G/N$ such that $gN \in f$.  Since every closed subgroup of $K$ of finite index contains a closed normal subgroup of $K$ of finite index, it can easily be seen that $\alpha$ is an isomorphism of topological groups.  Now $\theta|_K$ is the continuous map $\iota\alpha\inv\pi$, where $\pi: K \rightarrow \hat{K}$ is the profinite completion of $K$ and $\iota$ is the inclusion map of $\overline{\theta(K)}$ into $\hat{G}_K$.

Now we must show that property (ii) holds.  Let $R$ and $\phi$ be as given and let $L = \overline{\phi(K)}$.  Given $f \in \hat{G}_K$, define
\[ \hat{f} = \bigcap \{\overline{\phi(gN)} \mid g \in G, N \in \mcF(K), gN \in f\}.\]
Consider first $\hat{1} =  \bigcap \{ \overline{\phi(N)} \mid N \in \mcF(K)\}$.  By continuity of $\theta|_K$ we see that every closed subgroup $P$ of $L$ of finite index contains $\phi(N)$ for some $N \in \mcF(K)$, namely $N = \phi\inv(P) \cap K$.  Since $L$ is profinite, this ensures that $\hat{1}$ is the trivial group.  For general $f$ we conclude that $|\hat{f}|\le 1$ by the construction of $\hat{f}$ as an intersection of cosets; since $\hat{f}$ is the intersection of compact sets, any finite intersection of which is non-empty, in fact $|\hat{f}| =1$.  We define $\psi: \hat{G}_K \rightarrow R$ by setting $\psi(f)$ to be the unique element of $\hat{f}$; it is clear that $\psi$ is a homomorphism and that $\phi = \psi\theta$.  To see that $\psi$ is continuous, it suffices to consider the base of neighbourhoods of the identity in $R$ consisting of open subgroups $P$ of $L$; in this case $\psi\inv(P)$ is a union of cosets of the open subgroup $B(\phi^{-1}(P) \cap K)$ of $\hat{G}_K$.  The equation $\phi = \psi\theta$ determines the restriction of $\psi$ to $\theta(G)$ and hence determines $\phi$ uniquely as a continuous map, since $\theta(G)$ is dense in $\hat{G}_K$.

A standard universal property argument shows that $\theta$ is essentially unique.  Suppose that $\phi: G \rightarrow R$ is another localised profinite completion.  Then $\phi|_K$ is continuous and $\overline{\phi(K)}$ is profinite by property (i), so by property (ii) there is a continuous homomorphism $\psi: \hat{G}_K \rightarrow R$ such that $\phi = \psi\theta$.  Similarly there is a continuous homomorphism $\upsilon: R \rightarrow \hat{G}_K$ such that $\theta = \upsilon\phi = \upsilon\psi\theta$.  By property (ii), in fact $\upsilon\psi$ is the unique continuous endomorphism $\eta$ of $\hat{G}_K$ such that $\theta = \eta\theta$, so $\upsilon\psi = \id_{\hat{G}_K}$.  Similarly $\psi\upsilon = \id_R$, so $\upsilon$ and $\psi$ are mutually inverse isomorphisms of topological groups.\end{proof}

\end{document}